\documentclass{amsart}

\usepackage{amsmath}
\usepackage{amsfonts}
\usepackage[T1]{fontenc}
\usepackage[latin1]{inputenc}
\usepackage{amsthm}

\newtheorem{theorem}{Theorem}[section]
\newtheorem{lemma}[theorem]{Lemma}

\newtheorem{example}{Example}[section]

\theoremstyle{definition}
\newtheorem{definition}{Definition}

\theoremstyle{remark}
\newtheorem{remark}{Remark}[section]

\usepackage{verbatim}

\title[A maximum principle, f.B.M., infinite horizon]{A maximum principle for fractional diffusion processes with infinite horizon}

\author{Sven Haadem}

\address{Center of Mathematics for Applications (CMA), University of Oslo,Box 1053 Blindern, N-0316 Oslo, Norway}
\email{sven.haadem@cma.uio.no}

\date{15 June 2012}



\begin{document}

\begin{abstract}
We prove a maximum principle for the problem of optimal control for a fractional diffusion with infinite horizon. Further, we show existence of fractional backward stochastic differential equations on infinite horizon. Finally, we illustrate our findings with an example.
\end{abstract}

\maketitle

\section{Introduction}
In this paper we consider a control problem with respect to a given performance functional:
\begin{align*}
J(u) = E\left[ \int_0^{\infty} f(t, X(t),u(t), \omega)dt\right],
\end{align*}
where $X(t)$ is a controlled fractional diffusion and $u(t)$ is the control process. We allow for the case when the controller only has access to the partial information $\mathcal{E}_t$ at time t. Thus, we have a infinite horizon problem with partial information. Infinite-horizon optimal control problems arise in many fields of economics, in particular in models of economic growth. Note that because of the general nature of the partial information filtration $\mathcal{E}_t$, we cannot use stochastic dynamic programming and the Hamilton-Jacobi-Bellman (HJB) equation to solve the above optimization problem. Thus our problem is different from partial observation control problems. 
\newline\newline
In the deterministic case the maximum principle by Pontryagin (1962) has been extended to infinite-horizon problems, but transversality conditions have not been taken into account in the litrature in general. 
The 'natural' transversality condition in the infinite case would be a zero limit condition, meaning in the economic sense that one more unit of good at the limit gives no additional value. But this property is not necessarily verified. In fact \cite{halkin} provides a counterexample for a 'natural' extension of the finite-horizon transversality conditions. Thus some care is needed in the infinite horizon case. This paper contains a extension of the maximum principle in \cite{BHOS} and the existence results of fractional backward stochastic differential equations found in \cite{HP} to infinite horizon.
\newline\newline
In the case of Brownian motion there have been several paper on infinite-horizon, see e.g. \cite{HOP}, and \cite{HMOP} and \cite{AHOP} for the delay case. We are going to extend the results in \cite{HOP} to the case of a controlled diffusion process driven by a noise with memory. Here the driving noise in our controlled system with memory is modelled by a fractional Brownian motion, $B_t^H$, $t \geq 0$, where the Hurst parameter $h \in (\frac{1}{2},1)$. By solving this problem we establish a sufficient stochastic maximum principle. The latter result requires the construction of unique solutions of fractional noide driven backward stochastic differential equations (fBSDE's) of the form
\begin{align*}
\begin{cases}
 dp(t) = -f(t,\eta_t,y_t,z_t)dt - z_tdB_t^H, \\
 \underset{t \to \infty}{\lim} p(t) = 0.
\end{cases}
\end{align*}
As for the theory of BSDE's, which has gained more and more importance in stochastic control and mathematical finance, we refer the reader to the seminal paper by Pardoux and Peng \cite{PP}.
For some technical background to fractional Brownian motion, see \cite{Bender} and \cite{Bernt3}. The latter one will be a basis for much of our framework. For a comprehensive background to the mathematical structures used, see \cite{BK}.

\section{Preliminaries}
Let $B_t^H$, $t \geq 0$, be a fractional Brownian motion with Hurst index $H \in (\frac{1}{2},1)$ on the probability space $(\Omega,\mathcal{F}^H,\mathbf{P}^H)$. We assume $(\Omega,\mathcal{F}^H,\mathbf{P}^H)$ is endowed with the natural filtration $\mathcal{F}^H_t$ of $B^H$, where $\mathcal{F}^H = \vee_{t\geq0} \mathcal{F}^H_t$.

In the following we aim at introducing stochastic integrals with repsect to $B_t^H$ by using techniques from Gaussian white noise analysis.
For a general introduction to White noise and Hida distributions in the Brownian motion case, the reader may consult the excellent books \cite{Bernt5}, \cite{Kuo}, \cite{Obata} and of course the classical book \cite{Hida}. For a comprehensive explanation of the contruction of the fractional Brownian motion, where many of the following theorems and lemmas are found, we refer to \cite{Bernt3}.

Let $1 > H > \frac{1}{2}$, and  
\[
 \phi(s,t) = \phi_{H}(s,t):= H(2H-1)|s-t|^{2H-2}. 
\]
Denote by $\mathcal{S}(\mathbb{R})$ the Schwartz space of rapidly decreasing smooth functions on $\mathbb{R}$, and for $f,g \in \mathcal{S}(\mathbb{R})$, define
\begin{align*}
  \langle f,g \rangle_{H,t} := \int_0^t\int_0^t f(s)g(r) \phi(s,r)dsdr,
\end{align*}
and
\begin{align*}
  \langle f,g \rangle_H := \int_{\mathbb{R}}\int_{\mathbb{R}} f(s)g(t) \phi(s,t)dsdt.
\end{align*}
Now consider the completion of  $\mathcal{S}(\mathbb{R})$ under this inner product and denote the resulting separable Hilbert space by
\begin{align}\label{eq:L2}
 L^2_{\phi}.
\end{align}
Let $ \Omega = \mathcal{S}'(\mathbb{R})$ be the topological dual of $\mathcal{S}(\mathbb{R})$ (the space of tempered distributions). By Bochner-Minlos theorem or a abstract Hilbert space argument we have that there exists a probability measure $\mathbb{P}^H$ on the Borel $\sigma$-algebra ,$\mathcal{B}(\Omega)$, such that
\begin{align*}
 \int_{\Omega} e^{i\langle \omega,f \rangle}d\mathbb{P}^H(\omega) = e^{-\frac{1}{2}\lVert f\rVert^2_H},
\end{align*}
for all $f \in \mathcal{S}(\mathbb{R})$, where $\langle \cdot,\cdot \rangle$ is the dual pairing between $\mathcal{S}(\mathbb{R})$ and $\mathcal{S}'(\mathbb{R})$. It is now easy to see from a Fourier argument that
\begin{align*}
  E_ {\mathbb{P}^H}[\langle \cdot, f \rangle] := E[\langle \cdot, f \rangle] = 0,
\end{align*}
and
\begin{align*}
 E[\langle \cdot, f \rangle^2] = \lVert f\rVert^2_H.
\end{align*}
Now we define the fractional Brownian motion as
\begin{align*}
 B_t^H = B^H(t,\omega) = \langle \omega , \mathbf{1}_{[0,t]}(\cdot) \rangle,
\end{align*}
by extension, which is an element of $L^2(\mathbb{P}^H)$ for each $t$.
Let $\mathcal{J} = (\mathbb{N}_{0}^{\mathbb{N}})_c$ be the set of all finite sequences $\alpha = (\alpha_1, \ldots,\alpha_m)$ of nonnegative integers. For $\alpha = (\alpha_1, \ldots,\alpha_m)$, let
\begin{align*}
 \mathcal{H}_{\alpha}(\omega) := h_{\alpha_1}(\langle \omega, e_1 \rangle) \cdots h_{\alpha_m}(\langle \omega, e_m \rangle),
\end{align*}
where $h_n$ is the n-th Hermite polynomial (see \cite{Bernt5}) and $\{e_n\}_{n\geq 0}$ a orthonormal basis of $L_{\phi}^2(\mathbb{R})$.
As in the case of Brownian motion we obtain a Wiener-It$\bar{o}$ chaos expansion.
\begin{theorem}
 let $F \in L^2(\mathbb{P}^H)$. Then there exist unique $C_{\alpha}$, $\alpha \in \mathcal{J}$ such that
\begin{align*}
 F(\omega) = \sum_{\alpha \in \mathcal{J}}c_{\alpha} \mathcal{H}_{\alpha}(\omega),
\end{align*}
where the convergence is in $F \in L^2(\mathbb{P}^H)$. We also have that
\begin{align*}
 \lVert F \rVert^2_H = \sum_{\alpha \in \mathcal{J}} \alpha! c^2_{\alpha},
\end{align*}
where $\alpha! := \alpha_1\cdots \alpha_m$.

\end{theorem}
See \cite{Bernt5}, Theorem 3.1.8 for a proof.
We are now ready to define the fractional Hida test function and distribution spaces.
\begin{definition}[The fractional Hida test function space]
Let $(\mathcal{S})_H$ be the set of all $\psi(\omega) = \sum_{\alpha \in \mathcal{J}} a_{\alpha} \mathcal{H}_{\alpha}(\omega) \in L^2(\mathbb{P}^H)$ such that
\begin{align*}
 \lvert \psi \rVert^2_{H,k} = \sum_{\alpha \in \mathcal{J}} \alpha! a_{\alpha}^2 (2\mathbb{N})^{k\alpha},
\end{align*}
where 
\begin{align*}
(2\mathbb{N})^{\alpha} = \underset{j}{\prod} (2j)^{\gamma_j}.
\end{align*}
\end{definition}

\begin{definition}[The fractional Hida distribution space]
Let $(\mathcal{S})^{*}_H$ be the set of all formal expansions
\begin{align*}
G(\omega) = \sum_{\beta \in \mathcal{J}} b_{\beta} \mathcal{H}_{\beta}(\omega) \in L^2(\mathbb{P}^H)
\end{align*}
such that
\begin{align*}
 \lvert G \rVert^2_{H,-q} = \sum_{\beta \in \mathcal{J}} \beta! b_{\alpha}^2 (2\mathbb{N})^{-q\beta},
\end{align*}
for some $q \in \mathbb{N}$.
\end{definition}
Let $(\mathcal{S})_H$ be equiped with the projective topology and $(\mathcal{S})^{*}_H$ with the inductive topology. Then $(\mathcal{S})^{*}_H$ is isomorph (in the category of topological vector spaces) to the topological dual of $(\mathcal{S})_H$ with action given by
\begin{align*}
 \ll G,\psi\gg := \langle G,\psi\rangle_{(\mathcal{S})_H} := \sum_{\alpha \in \mathcal{J}} \alpha!  a_{\alpha} b_{\alpha}.
\end{align*}
\begin{definition}
Let
\begin{align*}
F(\omega) = \sum_{\alpha \in \mathcal{J}} a_{\alpha} \mathcal{H}_{\alpha}(\omega) 
\end{align*}
and
\begin{align*}
G(\omega) = \sum_{\beta \in \mathcal{J}} b_{\beta} \mathcal{H}_{\beta}(\omega) 
\end{align*}
belong to $(\mathcal{S})^{*}_H$. Then we define the Wick product $F\diamond G$ by
\begin{align*}
F\diamond G(\omega) = \sum_{\alpha,\beta \in \mathcal{J}} a_{\alpha}b_{\beta} \mathcal{H}_{\alpha+\beta}(\omega) 
\end{align*}
\end{definition}
Further, let
\[
 D_t^{\phi}F = \int_{\mathbb{R}} D_s^H F \phi(t,s)ds
\]
denote the Malliavin $\phi$-derivative of F (see \cite{Bernt3} Definition 3.5.1).

\begin{definition}[Fractional Wick-It$\bar{o}$-Skorohod integral (fWIS)]
Suppose $Y: \mathbb{R} \to (\mathcal{S})^{*}_H$ is a function such that $Y(t) \diamond W^H(t)$ is Bochner (or weaker Pettis) integrable in $(\mathcal{S})^{*}_H$. Then we define the fWIS-integral, $\int_{\mathbb{R}} Y(t)dB_t^H$, as
\begin{align*}
 \int_{\mathbb{R}} Y(t)dB_t^H := \int_{\mathbb{R}} Y(t)\diamond W_t^Hdt
\end{align*}

\end{definition}

\section{Fractional Wick-It$\bar{o}$-Skorohod integral in $L^2$}

The integral defined above is an stochastic distribution, but we would like to work in $L^2$ so we need to do some work to ensure that. To achieve this, we extend the construction of the fractional Brownian integral in \cite{Bernt3} to infinite horizon.

\subsection{Integral of simple functions}
Consider a partition $\pi_n$ of $[o,n]$, $\pi:0=t_0 < t_1 < \ldots <t_n=n$.
We now define the integral
\begin{align*}
 \int_0^{\infty}F^n_sdB^H_s,
\end{align*}
for a a simple function $F^n \in L^2(\mathbb{P}^H)$ of the form
\begin{align*}
F^n(t,\omega) = \sum_{i=0}^{n} F^n_i(\omega) \mathbf{1}_{[t_i,t_{i+1})}(t), 
\end{align*}
as
\begin{align*}
\int_0^{\infty}F^n_sdB^H_s := S(F,\pi_n) := \sum_{i=o}^{n-1}F_{t_{i}}\diamond(B^H_{t_{i+1}} - B^H_{t_{i}}),
\end{align*}
if $\lVert f \rVert _{\mathcal{L}_{\phi}^{1,2}} < \infty $ where
\begin{align*}
 \lvert f \rvert _{\mathcal{L}_{\phi}^{1,2}} := E\bigg[ \int_{\mathbb{R}} \int_{\mathbb{R}} f(s)f(t) \phi(s,t)dsdt 
+ (\int_{\mathbb{R}} D_t^{\phi}f(t)dt)^2\bigg],
\end{align*}
Using that for $F,G \in L^2(\mathbb{P}^H)$ we have that $E[F\diamond G] = E[F]E[G]$ and
\begin{align*}
 E[S(F,\pi_n)] &=  \sum_{i=0}^{n} E \bigg[F^n_i(\omega) \diamond(B^H_{t_{i+1}} - B^H_{t_{i}})\bigg]\\
&=  \sum_{i=0}^{n} E \bigg[F^n_i(\omega) E(B^H_{t_{i+1}} - B^H_{t_{i}}) \bigg] = 0
\end{align*}
Further we have that
\begin{align*}
&E \bigg[(F^n_i(\omega) \diamond(B^H_{t_{i+1}} - B^H_{t_{i}}))(F^n_j(\omega) \diamond(B^H_{t_{j+1}} - B^H_{t_{j}}))\bigg]\\
&=  E\Bigg[ \int_{t_{i}}^{t_{i+1}} \int_{t_{j}}^{{t_{j+1}}} D^{\phi}_sF^n_{t_{i}}D^{\phi}_tF^n_{t_{i}}dtds + F^n_{t_{i}}F^n_{t_{j}}\int_{t_{i}}^{t_{i+1}} \int_{t_{j}}^{{t_{j+1}}}\phi(s,t)dsdt\Bigg],
\end{align*}
so that
\begin{align*}
 E[S(F^n,\pi_n)^2] &= \sum_{i,j=0} E\Bigg[ \int_{t_{i}}^{t_{i+1}} \int_{t_{j}}^{{t_{j+1}}} D^{\phi}_sF^n_{t_{i}}D^{\phi}_tF^n_{t_{i}}dtds\\
&+ F^n_{t_{i}}F^n_{t_{j}}\int_{t_{i}}^{t_{i+1}} \int_{t_{j}}^{{t_{j+1}}}\phi(s,t)dsdt\Bigg].
\end{align*}

\subsection{FWIS integral for general stochastic functions }
Define
\begin{align*}
|\pi_n | := \max_i (t_{i+1} - t_i)
\end{align*}
 and $f_t^{\pi} := f_{t_i}$ if $t_i \leq t \leq t_{i+1}$. Assume that  $E[\lVert f^{\pi_n} - f \rVert^2_H] \to 0$ and
\begin{align*}
 E[S(F,\pi_n)^2] &= \sum_{i,j=0} E\bigg[ \int_{t_{i}}^{t_{i+1}} \int_{t_{j}}^{{t_{j+1}}} D^{\phi}_sF_{t_{i}}D^{\phi}_tF_{t_{i}}dtds\\
&+ F_{t_{i}}F_{t_{j}}\int_{t_{i}}^{t_{i+1}} \int_{t_{j}}^{{t_{j+1}}}\phi(s,t)dsdt\bigg]
\end{align*}
converges to 0 as $|\pi_n| \to 0$. 
For a sequence of partitions $\{\pi_n\}$ such that $|\pi_n | \to 0$ then $S(F,\pi_n)$ is a Cauchy sequence in $L^2(\mathbb{P}^{H})$. The limit in $L^2(\mathbb{P}^{H})$ is 
\begin{align}\label{eq:def.int}
 \int_0^{\infty}f(s)dB^H(s) := \underset{|\pi | \to 0}{\lim} \sum_{k=0}^{k-1}f_{t_i} \diamond (B^H_{t_{i+1}} - B^H_{t_{i}}),
\end{align}
so that
\begin{align*}
 E\Bigg[\bigg|\int_0^{\infty}f(s)dB^H(s)\bigg|^2\Bigg] = E\bigg[ \int_{0}^{\infty} \int_{0}^{\infty} D^{\phi}_sF_{t_{i}}D^{\phi}_tF_{t_{i}}dtds + \lVert f \rVert^2_H\bigg]
\end{align*}

\begin{definition}[$\mathcal{L}_{\psi}(0,\infty)$ ]
Let $\mathcal{L}_{\phi}(0,\infty)$ be the family of stochastic prosesses, $f$, on $[0,\infty)$ with the following properties:
\begin{enumerate}
 \item $E\lVert f \rVert^2_H < \infty$, 
 \item $f$ is $\phi$-differentiable,
 \item the trace of $D_s^{\phi}f_t$, $0\leq s\leq t < \infty$ exists,
 \item $E[\int_0^{\infty}\int_0^{\infty} |D_s^{\phi}f_t|^2dsdt] < \infty$
 \item and for each sequence of partitions $(\pi_n, n \in \mathbf{N})$ such that $|\pi_n| \to 0$ as $n \to \infty$ we have that
\begin{align*}
\sum_{i,j=0} E[ \int_{t_{i}}^{t_{i+1}} \int_{t_{j}}^{{t_{j+1}}} |D^{\phi}_sF^{\pi}_{t_{i}^n}D^{\phi}_tF^{\pi}_{t_{i}^n} - D^{\phi}_sF^{\pi}_{t_{i}^n}D^{\phi}_tF^{\pi}_{t_{i}^n}|dtds 
\end{align*}
and
\begin{align*}
 E[\lVert F^{\pi} - F\rVert^2_{H}]
\end{align*}
tend to $0$ as $n \to \infty$, where $\pi_n : = 0=t_0^n < t_1^n < \ldots < t_{n-1}^n < t_n^n = n$.
 
\end{enumerate}
\end{definition}

We summarize the above in a theorem.
\begin{theorem}[Properties of the integral]
Let $\{f_t\}_{t \geq 0}$ be a stochastic process such that $f \in \mathcal{L}_{\psi}(0,\infty)$.
The limit \eqref{eq:def.int} exists and satisfies;
\begin{align*}
E[ (\int_0^{\infty} f(s)dB^H(s))^2] =  E[\int_0^{\infty}\int_0^{\infty} D^{\phi}(s)f(t)D^{\phi}(t)f(s)dsdt + \lVert f \rVert_{H} ]
\end{align*}
and
\begin{align*}
E[ \int_0^{\infty} f(s)dB^H(s)] = 0. 
\end{align*}

\end{theorem}

\subsection{It$\bar{o}$s formula}
In this section we present It$\bar{o}$s formula for fractional Brownian motion.
Fist let us look at the $\phi-derivative$:
\begin{theorem}
 Let $(F_t, t \in [0,\infty))$ be a stochastic process in $\mathcal{L}_{\phi}([0,\infty))$ and $\underset{0 \leq s \leq \infty}{\sup} E[|D_s^{\phi}F_s|] < \infty$ and let $\eta_t = \int_0^t G_udu  + \int_0^t F_u dB_u^H$. Then for $s,t >0$
\[
 D_s^{\phi} \eta_t = \int_0^t D_s^{\phi}G_u du  + \int_0^t D_s^{\phi}F_u dB_u^H + \int_0^t F_u\phi(s,u)du, 
\]
a.s.
\end{theorem}
For a proof see \cite{Bernt3}. Now let we arrive at It$\bar{o}$s formula:
\begin{theorem}[It$\bar{o}$ formula]
 Let $\eta_t = \eta_0 + \int_0^t G_udu + \int_0^t F_u dB_u^H$, where $(F_t, t \in [0,\infty))$ is a stochastic process in $\mathcal{L}_{\phi}([0,\infty))$. Assume there is an $\alpha > 1-H$ such that
\[
 E[|F_u - F_v|^2] \leq C|u-v|^{2\alpha},
\]
where $|u-v| \leq \delta$ for some $\delta >0$ and
\[
 \underset{0\leq u,v \leq t, |u-v| \to 0}{\lim} E[|D_u^{\phi}(F_u - F_v|^2] = 0.
\]
Let $f:\mathbf{R}_{+} \times \mathbf{R} \to \mathbf{R}$ be a function with continuous derivative in the first variable and twice differentiable with continuous first and second derivatives in the second variable. Assume all derivatives are bounded. Moreover, assume $E[\int_0^{\infty} |F_s D_s^{\phi} \eta_s|ds] < \infty$ and $(f^{'}(s,\eta_s)F_s)$ is in $\mathcal{L}_{\phi}([0,\infty)$. Then
\begin{align*}
f(t,\eta_t) &= f(0,0) + \int_0^t \frac{\partial f}{\partial s}(s,\eta_s)ds + \int_0^t \frac{\partial f}{\partial x}(s,\eta_s)G_sds + \int_0^t \frac{\partial f}{\partial x}(s,\eta_s)F_s dB^H_s\\
&+ \int_0^t \frac{\partial^2 f}{\partial x^2}(s,\eta_s)F_s D_s^{\phi}\eta_s ds
\end{align*}
a.s.
\end{theorem}
For a proof see \cite{Bernt3}.

\subsection{Fractional Clark-Hausmann-Ocone theorem}
In this section we briefly recall a Clark-Hausman-Cone theorem for fractional Brownian motion on a distribution space  $L^2(\mathbb{P}^H) \subset \mathcal{G}^{*} \subset (\mathcal{S})^*$, which we want to employ in Section 4. Let us first give a construction of $\mathcal{G}^*$ to which the operator $D_t$ will be extended.
\begin{enumerate}
 \item[] 
\begin{definition}[\cite{PT}, \cite{AOPU},\cite{Bernt3}]
Let $k \in \mathbb{N}_0$. We say that a random variable with chaos expansion
\begin{align*}
 \psi = \sum_{n=0}^{\infty} \int_{\mathbb{R}^n_{+}} f_n d(B^H)^{\otimes n}(t),
\end{align*}
where $f_n \in \hat{L}^2_H(\mathbb{R}^n_{+})$ belongs to the space $\mathcal{G}_k = \mathcal{G}_k(\mathbb{P}^H)$ if
\begin{align*}
 \lVert \psi \rVert^2_{\mathcal{G}_k} : = \sum_{n=0}^{\infty} n! \lVert f_n \rVert_{L^2_{\phi}(\mathbb{R}^n_{+})} e^{2kn} < \infty,
\end{align*}
where $\lVert \cdot \rVert_{L^2_{\phi}(\mathbb{R}^n_{+})}$ is the completion of $\mathcal{S}((\mathbb{R}^n_{+})$ as in \eqref{eq:L2}.
Now, letting
\begin{align*}
\mathcal{G} = \mathcal{G}(\mathbb{P}^H) = \underset{k \geq 0}{\cap} \mathcal{G}_k(\mathbb{P}^H),
\end{align*}
and equip $\mathcal{G}$ with the projective topology.

\item
Let $q \in \mathbb{N}_0$. We say that a function
\begin{align*}
 G = \sum_{n=0}^{\infty} \int_{\mathbb{R}^n_{+}} g_n d(B^H)^{\otimes n}(t),
\end{align*}
where $g_n \in \hat{L}^2_H(\mathbb{R}^n_{+})$ belongs to the space $\mathcal{G}_{-q} = \mathcal{G}_{-q}(\mathbb{P}^H)$ if
\begin{align*}
 \lVert G \rVert^2_{\mathcal{G}_{-q}} : = \sum_{n=0}^{\infty} n! \lVert g_n \rVert_{L^2_{\phi}(\mathbb{R}^n_{+})} e^{-2qn} < \infty,
\end{align*}
where $\lVert f_n \rVert_{L^2_{\phi}(\mathbb{R}^n_{+})}$ is the completion of $\mathcal{S}((\mathbb{R}^n_{+})$ as in \eqref{eq:L2}.
Now, letting
\begin{align*}
\mathcal{G}^{*} = \mathcal{G}^{*}(\mathbb{P}^H) = \underset{q \geq 0}{\cup} \mathcal{G}_{-q}(\mathbb{P}^H),
\end{align*}
and equip $\mathcal{G}$ with the inductive topology. Then $\mathcal{G}^{*}$ is the dual of $\mathcal{G}$, and the action of $G \in \mathcal{G}^{*}$ on $\psi \in \mathcal{G}$ is given by
\begin{align*}
 \ll G,\psi \gg = \sum_{n=0}^{\infty} n! ( g_n,f_n)_{L^2_{\phi}(\mathbb{R}^n_{+})}.
\end{align*}

\end{definition}

\end{enumerate}

We will also need a variant of the conditional expectation on $\mathcal{G}^*$, that is easier to work with.

\begin{definition}
\begin{enumerate}
 \item 
 Let 
\[
 G = \sum_{n=0}^{\infty} \int_{\mathbf{R}^n_{+}} g_n(s) d(B^H)^{\otimes n}(s) \in \mathcal{G}^{*}. 
\]
Then we define the fractional conditional expectation of $G$ with respect to $\mathcal{F}_t^H$ by
\[
\tilde{E} [F |\mathcal{F}_t^H ] = \int_{\mathbf{R}^n_{+}} g_n(s)\mathbf{1}_{\{0\leq s \leq t\}} d(B^H)^{\otimes n}(s)
\]
for $t \geq 0$. 
   \item
We say that $G \in \mathcal{G}^*$ is $\mathcal{F}_t^H$-measurable if 
\[
\tilde{E}[G|\mathcal{F}_t^H] = G, 
\]
for $t \geq 0$.
\end{enumerate}

\end{definition}

\begin{remark} 
The quasi-conditional expectation $E [ \cdot |\mathcal{F}_t^H ]$ was introduced in \cite{HO} to construct hedging strategies of financial claims.
\end{remark} 
An immediate consequence of the definition is the following, which is important in solving backward stochastic differential equations.
\begin{lemma} 
If $(f_u , 0 \leq u \leq T )$ is a real valued stochastic process such that
\begin{align*}
E[ \int_{0}^{T}\int_{0}^{T} \phi(s-t)|f_{t}||f_{s}|dtds + \int_{0}^{T} \int_{0}^{T}D^{\phi}_sf_tdsdt],
\end{align*}
and $\int_t^T f_u dB^H_t \in L^2(\mathbf{P}^H)$ then
\begin{align*}
 \tilde{E}[\int_0^T f_u dB^H_u | \mathcal{F}_t] =0.
\end{align*}
\end{lemma}

\begin{remark}
 Note that this also holds for $T = \infty$.
\end{remark}

\begin{proof}
Let $f_u = \sum_{n=0}^{\infty} I_n(g_n(u)) $ In be the chaos expansion of $f$ , where $g$ is a
function of n-variables. 
Then it is well known  that
\begin{align}
\int_0^T f_u dB^H_u =  \sum_{n=0}^{\infty} I_{n+1}(\tilde{g}_n(u)),
\end{align}
where $\tilde{g}_{n+1}$ is the symmetrization.
Now the result follows from the definition of the quasi-conditional expectation.
\end{proof}

\begin{lemma}[Clark-Hausmann-Ocone representation]
 Let $F \in L^2(\mathbf{P}^H)$ then
\begin{align*}
 F = E[F] + \int_0^{\infty} \tilde{E}[D_tF|\mathcal{F}_t^H]dB_t^H
\end{align*}
\end{lemma}
\begin{proof}
\begin{align*}
&= E[F] + \sum_{n=1}^{\infty} \int_{[0 \leq s_1,\cdots, s_n \leq \infty ]} f(s_1,\ldots,s_n) dB_{s_1}^H \cdots dB_{s_n}^H \\
&= E[F] + \sum_{n=1}^{\infty} \int_0^{\infty} ( \int_{[0\leq s_1,\cdots,s_{n-1} \leq t]} f(s_1,\ldots,s_{n-1},t)dB_{s_1}^H \cdots dB_{s_{n-1}}^H)dB_t^H \\
&= E[F] + \sum_{n=1}^{\infty} \int_0^{\infty} \tilde{E}[D_t I_n(f_n)|\mathcal{F}_t^H] dB_t^H \\
&= E[F] + \int_0^{\infty} \tilde{E}[D_tF|\mathcal{F}_t^H]dB_t^H.
\end{align*}
\end{proof}

\newpage

\section{Fractional backward stochastic differential equations (fBSDE)}
Let $B_t^H$, $t \geq 0$, be a fractional Brownian motion with Hurst index $H > \frac{1}{2}$ on the probability space $(\Omega, \mathbf{F}, \mathbb{P}^H)$ endowed with the natural filtration $\mathcal{F}_t^H$ of $B^H$ and $\mathcal{F} = \cup_{t \geq 0} \mathcal{F}_t^H$. Let $b: [0,\infty) \times \mathbb{R} \times \mathbb{R} \to \mathbb{R}$. Consider the problem of finding a $\mathcal{F}^H$-adapted processes $(p(t),q(t))$ such that
\begin{align*}
\begin{cases}
dp(t) &= b(t,p(t),q(t))dt + q(t)dB^H(t),\\ 
\underset{t \to \infty}{\lim} p(t) &= 0.
\end{cases}
\end{align*}
This is a infinite horizon fractional backward stochastic differential equation (ihfBSDE).

\subsection{Existence of general FBSDE}\label{Chap:FBSDE}

In this section we prove a result about existence and uniqueness of the solution $(Y(t),Z(t),K(t,\zeta))$ of  infinite horizon BSDEs of the form;
\begin{align}
dY(t) &= -g(t,Y(t),Z(t),K(t,\cdot))dt + Z(t) dB^H(t); 0\leq t\leq \tau, \label{eq.existence1}\\
\underset{t \to \tau}{\lim} Y(t) &= \xi(\tau)\mathbf{1}_{[0,\infty)}(\tau), \label{eq.existence2}
\end{align}
where $\tau \leq \infty$ is a given $\mathcal{F}_t$-stopping time, possibly infinite.
Our result is the fractional vesion of \cite{HOP}, and the infinite horizon version of \cite{HP}. See also  \cite{Pardoux},  \cite{Peng},  \cite{Yin}, \cite{Li}, \cite{Tang}, \cite{BSDE}, \cite{BBP} and \cite{Rong}, for the classical Brownian motion case.

Let $\eta_t = \eta_0 + b_t + \int_0^t \sigma_s dB_s^H$, where $\eta_0$ is a given constant,$b_t$ is a deterministic difierentiable function of $t$, and $\sigma_s$ is a deterministic continuous
function such that $\sigma_t$ exists for all $t$ and $\frac{d}{dt} \Vert \sigma \rVert_t$ exists and it is strictly positive. Let $\xi = h(\eta_{\tau})$ and a continuous $h$ .
Let 
\[
 p_t(x) = \frac{1}{\sqrt{2\pi t}}e^{-\frac{x^2}{2t}}.
\]
Let the semigroup $\{P_t\}_{t \geq 0}$ be given by
\[
 P-t f(x) = \int_{\mathbf{R}} p_t(x-y) f(y)dy.
\]
If f is continous we have that 
\[
 \frac{\partial }{\partial t} P_t f(x) =  \frac{1}{2} \frac{\partial^2 }{\partial x^2} P_t f(x).
\]
Using It$\bar{o}$, we get 
\[
 f(\eta_t) = P_{\lVert \sigma \rVert_t^2}f(\eta_0) + \int_0^T \frac{\partial }{\partial t}P_{\lVert \sigma \rVert_t^2- \lVert \sigma \rVert_s^2}f(\eta_s)\sigma(s)dB^H_s.
\]
Therefore
\[
 \hat{E}[f(\eta_t)|\mathcal{F}_t] = P_{\lVert \sigma \rVert_t^2}f(\eta_0) + \int_0^t \frac{\partial }{\partial t}P_{\lVert \sigma \rVert_t^2- \lVert \sigma \rVert_s^2}f(\eta_s)\sigma(s)dB^H_s.
\]

Define
\[
\mathcal{V}_{\infty} = \{Y(\cdot) = \phi(\cdot,\eta(\cdot); \phi(t, \eta_t ) \in C^{1,2}_t \text{ for all } t \in [0, \infty)\},
\]
where $C^{1,2}_t$ is the set of functions that are continuously differentiable with respect to t
and twice continuously differentiable with respect to x.
Let $\tilde{\mathcal{V}}_{\infty}$ be the completion of $\mathcal{V}_{\infty}$ under the norm
\[
 \lVert Y \rVert_{\gamma}^2 := \int_0^{\infty} e^{\lambda t} E[|Y_t|^2] dt = \int_0^{\infty} e^{\lambda t} E[|\phi(t,\eta_t)|^2] dt. 
\]

We require the folowing:
\begin{enumerate}
  \item For each $\tau \geq T > 0$
\[
 \underset{0 \leq s \leq T}{\inf} \frac{\hat{\sigma}_s}{\sigma_s} \geq c_0,
\]
for some posetive constant $c_0$, where
\[
 \hat{\sigma}_s = \int_0^s \phi(s,r)\sigma_r dr.
\]
   \item The function $g: \Omega  \times \mathbb{R}_{+} \times \mathbb{R} \times \mathbb{R}  \to \mathbb{R} $ is such that there exist real numbers $\mu,$ $\lambda$ and $K$, such that  $K >0$ and
\begin{align}\label{eq:lambda_req}
\lambda > 2\mu + 2K^2.
\end{align}
We also assume that the function $g$ satisfies the following requirements:
\begin{enumerate}
	\item $g(\cdot,y,z)$ is progessively measurable for all $y \in \mathbb{R},z \in \mathbb{R}$, and
\begin{align}
 &|g(t,y,z) - g(t,y,z')| \leq K  |z - z'|. \notag\\
\end{align}
	\item 
\begin{align}
\langle y-y',g(t,y,z,k) - g(t,y',z,k)\rangle \leq \mu |y-y'|^2
\end{align}
for all $y,y',z$ a.s.
	\item 
\begin{align}
E \int_0^{\tau} e^{\lambda t} |g(t,0,0)|^2dt < \infty.
\end{align}
	\item  Finaly we require that
\begin{align}
y \mapsto g(t,y,z),
\end{align}
is continuous for all $t,z$ a.s.
	\end{enumerate}
  \item 
Let $\eta_t = \eta_0 + b_t + \int_0^t \sigma_s dB_s^H$ as above. Assume the final condition $\xi$  is given by  $\xi = \eta_{\tau}$,  and $\xi_t = \hat{E}[\xi |\mathcal{F}_t]$ such that $E(e^{\lambda \tau}|\xi|^2) < \infty $ and 
\begin{align}
 E \int_0^{\tau} e^{\lambda t} |g(t,\xi_t,\sigma_t)|^2dt < \infty.
\end{align}
\end{enumerate}
A solution of the BSDE \eqref{eq.existence1}-\eqref{eq.existence2}, is a dual $(Y_t,Z_t)$ of progressively measurable processes with values in $\mathbb{R} \times \mathbb{R} $ s.t. $Z_t=0$ when $t > \tau$, 
\begin{enumerate}
	\item $E[ \int_0^{\tau}e^{\lambda s} D_t^{\phi}Z_s ds] < \infty$,
	\item  $Y_t = Y_{T \wedge \tau} + \int_{t \wedge \tau}^{T \wedge \tau} g_s ds - \int_{t \wedge \tau}^{T \wedge \tau} Z_s dB^H_s  $ for all deterministic $T < \infty$ and
	\item $Y_t = \xi$ on the set $\{t \geq \tau\}$.
\end{enumerate}
\begin{remark}[Infinite Horizon]
This incorperates the case where $\tau(\omega) = \infty$ on some set $A$ with $P(A) > 0$, possibly $P(A)=1$.
\end{remark}

\begin{theorem}[Existence and uniqueness]\label{thm:infinite solution}
Under the above conditions there exists a unique solution $(Y_t,Z_t)$ of the BSDE \eqref{eq.existence1}-\eqref{eq.existence2}, which satisfies the condition;
\begin{align}
&E [   \int_{0}^{\tau} e^{\lambda s} (|Y_s|^2 + |Z_s|^2ds]\notag \\ 
&\leq c E[e^{\lambda \tau}|\xi|^2 + \int_0^{\tau}e^{\lambda s}|g(s,0,0)|^2ds]\label{eq.Req},
\end{align}
for some positive number $c$.

\end{theorem}

\begin{proof} \textit{First, let us show uniqueness}: \newline 
Let $(Y,Z)$ and $(Y',Z')$ be two solutions satisfying \eqref{eq.Req} and let $(\bar{Y},\bar{Z}) = (Y-Y',Z-Z')$. From It$\bar{o}$'s Lemma we have that
\begin{align*}
&e^{\lambda t \wedge \tau} |\bar{Y}_{t \wedge \tau} |^2 \\
&\leq - \int_{t \wedge \tau}^{T \wedge \tau} e^{\lambda s}  \lambda |\bar{Y}_s|^2 ds
+ e^{\lambda s}|\bar{Y}_{T}|^2 
+ 2\int_{t \wedge \tau}^{T \wedge \tau} e^{\lambda s} \bar{Y}_sf(t,\bar{Y}_s,\bar{Z}_s)ds \\
&+  2 \int_{t \wedge \tau}^{T \wedge \tau} e^{\lambda s} D^{\phi}_s\bar{Y}_s\bar{Z}_sds 
-  2 \int_{t \wedge \tau}^{T \wedge \tau} e^{\lambda s} \bar{Y}_s\bar{Z}_s dB_s^H\\
&\leq  e^{\lambda s}|\bar{Y}_{T}|^2\\
&+ \int_{t \wedge \tau}^{T \wedge \tau} e^{\lambda s}\Bigg[  -2\lambda |\bar{Y}_s|^2  
+  \mu |\bar{Y}_s|^2  + K |\bar{Y}_s|\lvert \bar{Z}_s \rvert\Bigg]ds
+  2 \int_{t \wedge \tau}^{T \wedge \tau} e^{\lambda s} D^{\phi}_s\bar{Y}_s\bar{Z}_sds\\
&-  2 \int_{t \wedge \tau}^{T \wedge \tau} e^{\lambda s} \bar{Y}_s\bar{Z}_s dB_s^H\\
&\leq  e^{\lambda s}|\bar{Y}_{T}|^2\\
&+ \int_{t \wedge \tau}^{T \wedge \tau} e^{\lambda s}\Bigg[ - 2\lambda |\bar{Y}_s|^2 
+  \mu |\bar{Y}_s|^2  + K |\bar{Y}_s|\lvert \bar{Z}_s \rvert\Bigg]ds
-  2 c_0 \int_{t \wedge \tau}^{T \wedge \tau} e^{\lambda s} |\bar{Z}_s|^2ds\\
&-  2 \int_{t \wedge \tau}^{T \wedge \tau} e^{\lambda s} \bar{Y}_s\bar{Z}_s dB_s^H.\\
\end{align*}
The last inequality follows for the fact that for $Y_t = \tilde{\phi}(t,\eta_t)$ we have that $Z_t = -\sigma_t \tilde{\phi}_{x}(t,\eta_t)$ (see \cite{HP} Proposition 4.3)  so that
\begin{align*}
D^{\phi}_s Y_s &= \int_0^s \phi(s-r)D_r Y_s dr\\
&= \tilde{\phi}_{x}(s,\eta_s) \int_0^s \phi(s-r)\sigma_r dr\\
&= \hat{\sigma}_s \tilde{\phi}_{x}(s,\eta_s)\\ 
&= -\frac{\hat{\sigma}_s}{\sigma_s}Z_s.
\end{align*}
So 
\begin{align*}
 D^{\phi}_s\bar{Y}_s = -\frac{\hat{\sigma}_t}{\sigma_t}\bar{Z}_t.
\end{align*}
Then it follows that
\begin{align*}
&e^{\lambda t \wedge \tau} |\bar{Y}_{t \wedge \tau} |^2 + \int_{t \wedge \tau}^{T \wedge \tau} e^{\lambda s}2\Bigg[  \lambda |\bar{Y}_s|^2  +  c_0 \lvert \bar{Z}_s \rvert^2\Bigg]ds \\
&\leq  e^{\lambda (T \wedge \tau)}|\bar{Y}_{T}|^2 +  \int_{t \wedge \tau}^{T \wedge \tau} 2\Bigg[e^{\lambda s}  \mu |\bar{Y}_s|^2 ds +  K |\bar{Y}_s|\lvert \bar{Z}_s \rvert)\Bigg]ds\\
&-  2 \int_{t \wedge \tau}^{T \wedge \tau} e^{\lambda s} \bar{Y}_s\bar{Z}_s dB_s^H.\\
\end{align*}

Combining the above with the fact that $2ab \leq \frac{a^2}{2c_0} + 2c_0 b^2 $ and $\lambda > 2\mu +2K^2$, we deduce, that for $t < T$

\begin{align*}
&e^{\lambda t \wedge \tau} |\bar{Y}_{t \wedge \tau} |^2 \\
&\leq  e^{(T \wedge \tau)}|\bar{Y}_{T}|^2 -  2 \int_{t \wedge \tau}^{T \wedge \tau} e^{\lambda s} \bar{Y}_s\bar{Z}_s dB_s^H.
\end{align*}
Letting $T \to \infty$, so that we have, since $ E[\int e^{\lambda t} |\bar{Y}(t)|^2] < \infty$,
\begin{align*}
&e^{\lambda t \wedge \tau} |\bar{Y}_{t \wedge \tau} |^2 \\
&\leq  -  2 \int_{t \wedge \tau}^{\tau} e^{\lambda s} \bar{Y}_s\bar{Z}_s dB_s^H.\\
\end{align*}
Taking expectation the uniqueness follows.\newline\newline
\textit{Proof of existence:}\newline
For each $n \in \mathrm{N}$ we construct a solution $(Y_t^n,Z^n_t) $ of the BSDE 
\[
 Y_t^n = \xi + \int_{t \wedge \tau}^{n \wedge \tau} g(s,Y_s^n,Z_s^n)ds 
-\int_{t \wedge \tau}^{\tau}  Z^n_s dB_s  
\]
by letting $\{(Y_t^n,Z^n_t); 0 \leq t\leq n \}$  be defined as a solution of the following BSDE:
\[
 Y_t^n = \hat{E}[\xi|\mathcal{F}_n] + \int_{t }^{n } \mathbf{1}_{[0,\tau]}(s)g(s,Y_s^n,Z_s^n)ds
-\int_{t }^{n }  Z^n_s dB_s 
\]
for $0\leq t\leq n$ and $\{(Y_t^n,Z^n_t); t\geq n \}$ defined by
\[
 Y_t^n = \xi_t,
\]
and
\[
 Z^n_t = \sigma_t,
\]
for $t> n$. From Theorem 4.6 in \cite{HP} we we now that this finite horizon equation has a unique solution thanks to our requirements. Next, we find some a priori estimates for the sequence $(Y^n,Z^n,K^n)$. For any $\epsilon > 0$ and $\rho < 1$   we have for all $t \geq 0, y \in \mathbb{R}$, $z\in \mathbb{R}$  with $c= \frac{1}{\epsilon}$,
\begin{align*}
&2 \langle y,g(t,y,z)\rangle = 2 \langle y,g(t,y,z) - g(t,0,z) \rangle \\
&+ 2 \langle y,g(t,0,z) - g(t,0,0) \rangle + 2 \langle y,g(t,0,0) \rangle  \\
&\leq (2\mu + \frac{1}{\rho}K^2  + \epsilon)|y|^2
+ \rho | z_s |^2 \\
&+ c|g(t,0,0)|^2.
\end{align*}
From It$\bar{o}$'s Lemma and arguments given above, we have
\begin{align*}
e^{\lambda t \wedge \tau} |Y_{t \wedge \tau}^n |^2 &+ \int_{t \wedge \tau}^{\tau} e^{\lambda s} \left[ \bar{\lambda} |Y_s^n|^2 + \bar{\rho}|z_s |^2 \right] ds\\ 
&\leq e^{\lambda s}|\eta|^2 + c\int_{t \wedge \tau}^{\tau} e^{\lambda s} |g(s,0,0,0)|^2ds\\ 
&- 2\int_{t \wedge \tau}^{\tau} e^{\lambda s} <Y_s^n,Z_s^ndB_s>,
\end{align*}
with $\bar{\lambda} = \lambda - 2\mu - \frac{1}{\rho}K^2  - \epsilon > 0$ and $\bar{\rho} = 1-\rho >0$  . 
From this it follows that
\begin{align*}
&E \left[  \int_{s \wedge \tau}^{\tau} \Big[e^{\lambda r} (|Y_r^n|^2 + \parallel Z_r^n \parallel^2) \Big]dr\right]\\
&\leq C E\left[ e^{\lambda \tau}|\xi|^2 + \int_{s \wedge \tau}^{\tau} e^{\lambda r}|g(r,0,0)|^2dr \right].
\end{align*}
Let $m > n$ and define $\Delta Y_t := Y_t^m - Y_t^n$ and $\Delta Z_t := Z_t^m - Z_t^n$, so that for $n \leq t \leq m$,
\[
 \Delta Y_t = \int_{t \wedge \tau}^{m \wedge \tau} g(s,Y_s^m,Z_s^m) ds
-\int_{t \wedge \tau}^{m \wedge \tau} \Delta Z_s dB_s.
\]
It then follows that
\begin{align*}
&\int_{t \wedge \tau}^{m \wedge \tau} \Big\{e^{\lambda s} ( \lambda |\Delta Y_s|^2 +  | \Delta Z_s |^2) \\ 
\leq &e^{\lambda t \wedge \tau} |\Delta Y_{t \wedge \tau} |^2 + \int_{t \wedge \tau}^{m \wedge \tau} \Big\{e^{\lambda s} ( \lambda |\Delta Y_s|^2 +   | \Delta Z_s |^2) \\ 
& = \int_{t \wedge \tau}^{m \wedge \tau}  e^{\lambda s}  \Delta Y_s g(s,Y_s^m,Z_s^m,K_s^m) ds\\
&- 2 \int_{t \wedge \tau}^{m \wedge \tau}  e^{\lambda s} \Delta Y_s \Delta Z_s dB_s  \\
&2\leq e^{\lambda s}|\eta|^2 c\int_{t \wedge \tau}^{m \wedge \tau} e^{\lambda s} |g(s,0,0)|^2ds 
- 2\int_{t \wedge \tau}^{m \wedge \tau} e^{\lambda s} \Delta Y_s \Delta Z_s dB_s .
\end{align*}
From the same arguments as above
\begin{align*}
&E \Bigg[ \int_{n \wedge \tau}^{m \wedge \tau} \Big\{e^{\lambda s} (|\Delta Y_s|^2 + | \Delta Z_s |^2) \Big\}ds\Bigg]\\
&\leq 4E\left[\int_{n \wedge \tau}^{\tau} e^{\lambda s}|g(s,\xi,\sigma)|^2ds\right]. 
\end{align*}
The last term in the above equation goes to zero as $n \to \infty$. Now, for $t\leq n$
\begin{align*}
\Delta Y_t &= \Delta Y_n + \int_{t \wedge \tau}^{n \wedge \tau} \Big\{g(s,Y_s^m,Z_s^m) - g(s,Y_s^n,Z_s^n)\Big\} ds -\int_{t \wedge \tau}^{n \wedge\tau} \Delta Z_s dB_s. 
\end{align*}
Using the same argument as in the case of uniqueness, we have that
\begin{align*}
 E[e^{\lambda t \wedge \tau} |\Delta Y_{t\wedge \tau}|^2] 
\leq E[e^{\lambda t \wedge \tau} |\Delta Y_n|^2] 
\leq c E\left[ \int_{n \wedge \tau}^{\tau} e^{\lambda s}|g(s,\xi_s,\eta_s)|^2ds\right].
\end{align*}
It now follows that the sequence $(Y^n,Z^n)$ is Cauchy in the norm
\begin{align*}
\lVert (Y,Z) \rVert &:= E [   \int_{0}^{\tau} e^{\lambda s} (|Y_s|^2 + |  Z_s |^2)ds].
\end{align*}
So, we have that there is an unique solution to the BSDE \eqref{eq.existence1}-\eqref{eq.existence2}, which satisfies for all  $\lambda > 2\mu +2K^2$, the condition
\begin{align*}
&E \left[   \int_{0}^{\tau} e^{\lambda s} (|Y_s|^2 + |  Z_s |^2)ds \right]\\ 
&\leq c E\left[e^{\lambda \tau}|\xi|^2 + \int_0^{\tau}e^{\lambda s}|g(s,0,0)|^2ds\right].
\end{align*}

\end{proof}

\subsection{linear infinite horizon backward stochastic differential equations (LIHBSDE)}
For linear infinite horizon backward stochastic differential equations (LIHBSDE), we can give a constuctive proof of existence.
\begin{align}\label{eq.LIHFBSDE}
dp(t) &= [\alpha(t) + b(t)p(t) + c(t)q(t)]dt + q(t)d\hat{B}^H(t),\\ 
\underset{t \to \infty}{\lim} p(t) &= 0. \notag
\end{align}
where $b(t)$ and $c(t)$ are given continuous deterministic functions and $\alpha(t) = \alpha(t,\omega)$ is a given $\mathbf{F}^H$-adapted process such that $\int_0^{\infty}|\alpha(t,\omega)|dt < \infty$ almost surely.
By the fractional Girsanov theorem we can rewrite \eqref{eq.LIHFBSDE} as
\begin{align}\label{eq.LIHBSDE II}
 dp(t) = [\alpha(t) + b(t)p(t)]dt + q(t)dB^H(t),
\end{align}
where
\[
 \hat{B}^H(t) = B^H(t) + \int_0^t c(s) ds
\]
is a fBm under the probability measure $\hat{\mathbb{P}}^H$ on $\mathbf{F}^H$ defined by
\[
\frac{d\hat{\mathbb{P}}^H(\omega)}{\mathbb{P}^H(\omega)} = \exp^{\diamond}(-\langle \omega, \hat{c} \rangle) = \exp(-\int_0^{\infty} \hat{c}(s)dB^H(s) - \frac{1}{2} \lVert \hat{c} \rVert_H^2),
\]
where $\hat{c}$ is the the continuous function with supp $\hat{c} \subset [0,\infty)$ satisfying
\[
 \int_0^{\infty} \hat{c}(s)\phi(s,t) ds = c(t)
\]
and
\[
 \lVert \hat{c} \rVert_H^2 = \int_0^{\infty} \int_0^{\infty} \hat{c}(s)\hat{c}(t) \phi(s,t)ds dt.
\]
Let us now multiply \eqref{eq.LIHBSDE II}  by the integrating factor
\[
 \beta(t) := \exp(-\int_0^{t} b(s)ds),
\]
so that we get
\begin{align*}
 d(\beta(s)p(s) = \beta(s)\alpha(s)ds + \beta(s)q(s)d\hat{B}^H(s),
\end{align*}
or by integrating
\begin{align}\label{eq.int}
 \beta(t)p(t) = \int_t^{\infty} \beta(s)\alpha(s)ds + \int_t^{\infty}\beta(s)q(s)d\hat{B}^H(s).
\end{align}
Assume that
\begin{align*}
 \lVert \alpha \rVert_{\mathcal{L}_{\phi}}^2 : &= E_{\hat{\mathbb{P}}^H} [\int_0^{\infty}\int_0^{\infty} \alpha(s) \alpha(t) \phi(s,t) ds dt]\\
&+ E_{\hat{\mathbb{P}}^H} [\int_0^{\infty}\int_0^{\infty} \hat{D}^{\phi}(s)\alpha(t) \hat{D}^{\phi}(t)\alpha(s) ds dt] < \infty,
\end{align*}
where $\hat{D}^{\phi}$ denotes the $\phi$-derivative with respect to $\hat{B}^H$.
If we now apply the quasi-conditional expectation operator
\[
 \tilde{E}_{\hat{\mathbb{P}}^H}[\cdot | \mathcal{F}_t^H ]
\]
to \eqref{eq.int} we get
\begin{align}\label{eq.intSol}
 \beta(t)p(t) = \int_t^{\infty} \beta(s)\tilde{E}_{\hat{\mathbb{P}}^H}[\alpha(s) | \mathcal{F}_t^H ]ds 
\end{align}
From \eqref{eq.intSol} we get the solution
\begin{align}\label{eq:FBSDE}
 p(t) = \int_t^{\infty} \exp(- \int_t^s b(r)dr )\tilde{E}_{\hat{\mathbb{P}}^H}[\alpha(s) | \mathcal{F}_t^H ]ds.
\end{align}

\section{A infinite horizon maximum principle}
In this section we prove a maximum principle for systems driven by a fractional Brownian motion. For classical Brownian motion systems see e.g.\cite{Haussman}, \cite{Peng2}, \cite{Peng} and \cite{YZ} and the refrences therein for more information.
Assume we have an m-dimensional fractional Brownina motion, $B^H(t)$, with Hurst parameter $H = (H_1,H_2, \ldots,H_m)$. 
Let $X(t) = X^u(t)$ be a controlled fractional diffusion, described by the stochastic differential equation;
\begin{align}\label{eq.diffusion}
dX(t) &= b(t,X(t),u(t), \omega)dt + \sigma(t,X(t),u(t), \omega)dB^H(t) \notag\\
X(0) &= x \in \mathbb{R}^n.
\end{align}

Let
\begin{align*}
 \mathcal{E}_t \subset \mathcal{F}_t,
\end{align*}
be a given subfiltration, representing the information available to the controller at time $t; t \geq0$.
The process $u(t)$ is our control, assumed to be $\{\mathcal{E}_t\}_{t\geq0}$ adapted and with values in a set $U \subset \mathbb{R}^n$.
Let $\mathbb{A}$ be our family of $\mathcal{E}_t$-adapted controls.

Let $f:[0,\infty]\times \mathbb{R}^n \times U \times \Omega \rightarrow \mathbb{R}^n$ be adapted and assume that
\begin{align*}
E\left[ \int_0^{\infty} |f(t, X(t),u(t), \omega)|dt\right] < \infty \text{ for all } u\in \mathbb{A}.\\
\end{align*}

Then we define
\begin{align*}
J(u) = E\left[ \int_0^{\infty} f(t, X(t),u(t), \omega)dt\right]
\end{align*}
to be our performance functional.
\begin{definition}[Admissible pair]
Let $\mathcal{A}$ denote the $\mathcal{F}_t^H$-adapted processes $u: [0,\infty) \times \Omega \to U$ such that $X^u(t)$ exists and doesn't explode in $[0,\infty)$ and such that it satisfies the stochastic differential equation \eqref{eq.diffusion}. If $u \in \mathcal{A}$ and $X^u(t)$ is the corresponding state process, we call $(u,X^u(t))$ an admissible pair.
\end{definition}
We study the problem to find $\hat{u}\in \mathbb{A}$ such that
\begin{align*}
J(\hat{u}) = \sup_{u\in \mathbb{A}}J(u).
\end{align*}
If such $\hat{u} \in \mathcal{A}$ exists the $\hat{u}$ is called an optimal control and $(\hat{u},X^{\hat{u}}(t))$ is called an optimal pair.\\
Let $ C([0,\infty), \mathbb{R}^{n \times m})$ be the set of continuous functions from $[0,\infty)$ into $\mathbb{R}^{n \times m}$. We now define the Hamiltonian $H:[0,\infty) \times \mathbb{R}^n \times U \times \mathbb{R}^n \times  C([0,\infty), \mathbb{R}^{n \times m}) \rightarrow \mathbb{R}$, by
\begin{align}
H(t,x,u,p,q,\omega)  &= f(t,x,u, \omega) + b^{T}(t,x,u, \omega)p \notag \\
&+ \sum_{i=1}^{n} \sum_{j,k=1}^{m} \sigma_{i,k}(t,x,u) \int_0^{\infty} q_{i,k}(s)\phi_{H_k}(s,t)ds, \label{eq:Hamilton}
\end{align}
where $\phi_{H_k}(s,t)$ is defined as above.
For notational convenience we will in the rest of the paper suppress any $\omega$ in the notation.

The adjoint equation that arrise in the maximum principle is the unknown $\mathcal{F}_t$-predictable processes $(p(t),q(t))$ that satisfies the following stochastic differential equation;
\begin{align}\label{eq:adjoint} 
d\hat{p}(t) &= - \nabla_x H(t,\hat{X}(t),\hat{u}(t),\hat{p}(t),\hat{q}(t))dt + q(t)dB^H(t).
\end{align}
For the solution of equations of the form \eqref{eq:FBSDE} see Chapter \ref{Chap:FBSDE}.

\begin{theorem}[Partial Information, Infinite Horizon Fractional Maximum Principle]
Let $\hat{u} \in \mathcal{A}$ and let  $(\hat{p}(t),\hat{q}(t))$ be an associated solution to the  equation \eqref{eq:adjoint}. Define $I_4$ as
\[
I_4 := E[ \sum_i^{n} \sum_{j,k}^{m} \int_0^{\infty} \int_0^{\infty} D_{j,s}^{\phi_j}\{\sigma_{i,k}(t,X(t),u(t)) -  \sigma_{i,k}(t,\hat{X}(t),\hat{u}(t)) \}D_{k,t}^{\phi_k}\hat{q}_{i,j}(s)dtds].
\]
Assume that for all $u\in \mathcal{A}$ the following terminal condition holds:
\begin{align}
0 \leq E\left[ \overline{\lim_{t \to \infty}} [ \hat{p}(t)^{T}(X(t) - \hat{X}(t)) ]\right], \label{eq:lim}
\end{align}
and
\[
 I_4 \leq 0.
\]
Moreover, assume that \newline
$H(t,x,u,\hat{p}(t),\hat{q}(t))$ is concave in $x$ and $u$ and
\begin{align}
E\left[ H(t,\hat{X}(t),\hat{u}(t),\hat{p}(t),\hat{q}(t)) | \mathcal{E}_t\right] \notag \\
= \max_{u \in U} E\left[ H(t,\hat{X}(t),u,\hat{p}(t),\hat{q}(t)) | \mathcal{E}_t\right]. \label{eq:max}
\end{align}
and
\begin{align}\label{eq:derivL2}
&E\left[|\nabla_u H(t,\hat{X}(t),\hat{u}(t),\hat{p}(t),\hat{q}(t),\hat{r}(t,\cdot))|^2\right] < \infty.
\end{align}

Then we have that $\hat{u}(t)$ is optimal.
\end{theorem}

\begin{proof}
Let $I := E[ \int_0^{\infty} (f(t, X(t),u(t)) - f(t, \hat{X}(t),\hat{u}(t)))dt] = J(u) -J(\hat{u})$.
Then $I = I_{1} - I_{2} - I_{3} - I_{4}$, where
\begin{align}
I_{1} &:= E\Biggl[ \int_0^{\infty} (H(s,X(s),u(s),\hat{p}(s),\hat{q}(s)) \notag\\ &-H(t,\hat{X}(s),\hat{u}(t),\hat{p}(s),\hat{q}(s)))ds \Biggr] \label{eq:I11}, \\
I_{2} &:= E\left[ \int_0^{\infty} \hat{p}(s)^T(b(s)-\hat{b}(s))ds\right] \notag, \\
I_{3} &:= E\left[\int_0^{\infty} \int_0^{\infty} \sum_i^{n} \sum_{k}^{m} \Phi_{i,k}(s,\hat{X}(s),\hat{u}(s),X(s),u(s))\hat{q}_{i,k}(t) \phi_ {H_k}(s,t) ]dsdt\right] \notag, 
\end{align}
where $\Phi_{i,k}(s,\hat{X}(s),\hat{u}(s),X(s),u(s)):= \sigma_{i,k}(\hat{X}(s),\hat{u}(s)) - \sigma_{i,k}(X(s),u(s))$.
From concavity we get that
\begin{align}
&H(t,X(t),u(t),\hat{p}(t),\hat{q}(t)) - H(t,\hat{X}(t),\hat{u}(t),\hat{p}(t),\hat{q}(t))\\
&\leq \nabla_x H(t,\hat{X}(t),\hat{u}(t),\hat{p}(t),\hat{q}(t))^T(X(t) - \hat{X}(t)) \notag \\
&+ \nabla_u H(t,\hat{X}(t),\hat{u}(t),\hat{p}(t),\hat{q}(t))^T(u(t) - \hat{u}(t)). \label{eq:concave}
\end{align}

Then we have from \eqref{eq:max}, \eqref{eq:derivL2} and that $u(t)$ is adapted to $\mathcal{E}$,
\begin{align}
0 &\geq \nabla_u E\left[ H(t,\hat{X}(t),u,\hat{p}(t),\hat{q}(t)) | \mathcal{E}_t\right]^T_{u=\hat{u}(t)}(u(t) - \hat{u}(t)) \notag\\
&= E\left[ \nabla_u H(t,\hat{X}(t),\hat{u}(t),\hat{p}(t),\hat{q}(t))^T(u(t) - \hat{u}(t)) | \mathcal{E}_t\right]. \label{eq:equality}
\end{align}

Combining \eqref{eq:adjoint}, \eqref{eq:I11}, \eqref{eq:concave} and \eqref{eq:equality}
\begin{align*}
I_{1}^{\infty} &\leq E\left[ \int_0^{\infty} \nabla_x H(t,\hat{X}(s),\hat{u}(s),\hat{p}(s),\hat{q}(s))^T(X(s) - \hat{X}(s)) ds\right]\\
&= E\left[ \int_0^{\infty} (X(s) - \hat{X}(s))^T d\hat{p}(s) \right] =: -J_1.
\end{align*}
We see that
\begin{align*}
I_2 &= E\left[ \int_0^{\infty} \hat{p}(s)^T(b(s)-\hat{b}(s))ds\right]\\
&= E[\int_0^{\infty} \hat{p}(t)(d\hat{X}(t) - dX(t))]\\
&+ E[\int_0^{\infty} \hat{p}^T(t)\{\sigma_{i,k}(\hat{X}(s),\hat{u}(s)) - \sigma_{i,k}(X(s),u(s))\}dB^H(t)]\\
&= E[\int_0^{\infty} \hat{p}(t)(d\hat{X}(t) - dX(t))]
\end{align*}

Now, using \eqref{eq:lim} and Ito's formula
\begin{align*}
0 &\leq E\left[ \overline{\lim_{t \to \infty}} [ \hat{p}(t)^{T}(X(t) - \hat{X}(t)) ]\right] \\
&= E\left[ \int_0^{\infty} (X(s) - \hat{X}(s))^T d\hat{p}(s) \right] + E[\int_0^{\infty} \hat{p}(t)(d\hat{X}(t) - dX(t))]\\
&+  E\left[\int_0^{\infty} \int_0^{\infty} \sum_i^{n} \sum_{k}^{m} \Phi_{i,k}(s,\hat{X}(s),\hat{u}(s),X(s),u(s))\hat{q}_{i,k}(t) \phi_ {H_k}(s,t) ]dsdt\right]\\
&+ E[ \sum_i^{n} \sum_{j,k}^{m} \int_0^{\infty} \int_0^{\infty} D_{j,s}^{\phi_j}\{\sigma_{i,k}(t,X(t),u(t)) -  \sigma_{i,k}(t,\hat{X}(t),\hat{u}(t)) \}D_{k,t}^{\phi_k}\hat{q}_{i,j}(s)dtds]\\
&= I_{2}^{\infty} + J_{1}^{\infty} + I_{3}^{\infty} + I_{4}^{\infty}.
\end{align*}
Finally, combining the above we get
\begin{align*}
J(u) - J(\hat{u}) &\leq I_{1}^{\infty} - I_{2}^{\infty} - I_{3}^{\infty} - I_{4}^{\infty}\\
&\leq  -J_1^{\infty} - I_{2}^{\infty} - I_{3}^{\infty} - I_{4}^{\infty}\\
&\leq 0.
\end{align*}
This holds for all $u\in \mathcal{A}$ so the result follows.
\end{proof}

\newpage

We now give a infinite horizon version of the example in \cite{BHOS}.
\begin{example}[minimal variance problem]
Consider a  financial market driven by two independent fractional Brownian motions, $B_1(t) = B_1^{H_1}(t)$ and $B_2(t) = B_2^{H_2}(t)$ with $ H_1,H_2 \in (\frac{1}{2},1)$, defined as follows:
\begin{enumerate}
 \item[1:]
A Bond Price
\begin{equation*}
 \begin{cases}
  dS_0(t) = 0,\\
  S_0(0) = 1.
 \end{cases}
\end{equation*}

 \item[2:] Stock 1
\begin{equation*}
 \begin{cases}
  dS_1(t) = dB_1(t),\\
  S_1(0) = s_1.
 \end{cases}
\end{equation*}

 \item[3:] Stock 2 
\begin{equation*}
 \begin{cases}
  dS_2(t) = dB_1(t) + dB_2(t),\\
  S_2(0) = s_2.
 \end{cases}
\end{equation*}

\end{enumerate}
If $\phi(t) = (\phi_0(t),\phi_1(t),\phi_2(t)) \in \mathbb{R}^3$ is a portfolio ( the number of units of bond, stock 1 and stock 2, respectively, held at time $t$), then the corresponding value process is
\[
 V^{\theta}(t) = \theta(t)S(t) = \sum_{i= 0}^2 \theta_i(t)S_i(t).
\]
The protfolio is called self-financing if
\begin{align*}
dV^{\theta}(t) = \theta(t)dS(t) = \theta_1(t) dB_1(t) + \theta_2(t) (dB_1(t) + dB_2(t)).
\end{align*}
The market is called complete if any bounded $\mathcal{F}^H$-measurable random variable $F$, can be hedged, in the sense that there exist a self-financing portfolio $\theta(t)$ and an intial value $z \in \mathbb{R}$ such that
\begin{align*}
 F = z + \int_0^{\infty} \theta(t)dS(t),
\end{align*}
almost surely (see \cite{HO} for more details).
Let us consider the case where we are unable to trade n stock 1.  Let us say we want to stay ``close'' in some sense to $B_1(t)$ at all times. We let $\theta_2(t) = u(t)$ and consider ``close'' as small $L^2(\mu)$-difference. Find $z \in \mathbb{R}$ and admissible $u(t)$ such that
\begin{align*}
 J(z,u) := E [ \{ B_1(T) - (z + \int_0^{T} u(t)(dB_1(t) + dB_2(t))dt) \}^2 ] 
\end{align*}
is minimal for all $T \geq 0$. Its clear that $z = 0$ is optimal, so it remains to minimize
\begin{align*}
 J(z,u) := E [ \{  \int_0^{T} (u(t)-1)(dB_1(t) + dB_2(t))dt \}^2 ] 
\end{align*}
Let
\[
 J(u) = E\left[\int_0^{\infty} e^{-\rho t} \frac{1}{2}X^2(t) dt \right],
\]
where
\[
 dX(t) = u(t)dB_1(t) + (1-u(t))dB_2(t),
\]
and $\rho \geq 0$.
We have
\begin{align*}
 H(t,x,u,p,q) &= \frac{1}{2}e^{-\rho t}x^2\\
&+ u\left( \int_0^{\infty} q_1(s) \phi_1(s,t)ds + \int_0^{\infty} q_2(s) \phi_2(s,t)ds\right)\\
&+ \int_0^{\infty} q_1(s) \phi_1(s,t)ds.
\end{align*}
So that
\[
 \nabla_x H(t,x,u,p,q) =  x e^{-\rho t}x^2,
\]
and
\[
 \nabla_u H(t,x,u,p,q) =  \int_0^{\infty} q_1(s) \phi_1(s,t)ds + \int_0^{\infty} q_2(s) \phi_2(s,t)ds.
\]
It now follows that we have 
\begin{align*}
&dp(t) = e^{-\rho t}X(t)dt + q_1(t)dB_1(t) + q_2(t)dB_2(t).
\end{align*}
If we let $q_1(t) = \frac{1}{\rho}e^{-\rho t}u(t), q_2(t) = \frac{1}{\rho}e^{-\rho t}(1-u(t))$ then we have that
\begin{align*}
 &H(t,\hat{X}(t),v,\hat{p}(t),\hat{q}(t)) = \frac{1}{2}e^{-\rho t}\hat{X}^2(t)\\
&+ v\left( \int_0^{\infty} \frac{1}{\rho}e^{-\rho s}\hat{u}(s) \phi_1(s,t)ds + \int_0^{\infty} \frac{1}{\rho}e^{-\rho t}(1-\hat{u}(t))s) \phi_2(s,t)ds\right)\\
&+ \int_0^{\infty}\frac{1}{\rho}e^{-\rho s}\hat{u}(s) \phi_1(s,t)ds.
\end{align*}
Since we need that the maximum of this expression is attained at $v = \hat{u}$, is is a easy to see that
we must have
\[
 \int_0^{\infty} ( (1-\hat{u}(s))\phi_1(s,t) - \hat{u}(s) \phi_2(s,t))ds = 0.
\]
This is a symmetric Fredholm integral equation of the first kind and it is well known that it ha a unique solution $\hat{u}(t) \in L^2([0,\infty))$, see \cite{LF}. This $\hat{u}(t)$ satisfies all reqirements in our theorem and we also see from It$\bar{o}$'s lemma that 
\[
 p(t) = \frac{1}{\rho}e^{-\rho t} X(t),
\]
so that we have that
\[
\overline{\lim_{t \to \infty}} \hat{p}(t)(X(t) - \hat{X}(t)) = 0,
\]
almost surely.
This gives us that
\[
 \hat{u}(t),
\]
is a optimal control.
\end{example}

\newpage

\bibliographystyle{amsalpha}
\bibliography{references}

\end{document}